\documentclass{amsart}
\usepackage{amssymb, amsmath, mathrsfs, verbatim, url, bbm,marvosym}
\usepackage[all,cmtip]{xy}

\theoremstyle{plain}
\newtheorem{theorem}{Theorem}
\newtheorem{lemma}[theorem]{Lemma}
\newtheorem{proposition}[theorem]{Proposition}
\newtheorem{corollary}[theorem]{Corollary}

\theoremstyle{definition}
\newtheorem{definition}[theorem]{Definition}
\newtheorem{question}{Question}

\newcommand{\cl}{\mathsf{cl}}
\newcommand{\dom}{\mathsf{dom}}
\newcommand{\Ps}{\mathsf{(s)}}
\newcommand{\Psp}{\mathsf{(s)^+}}
\newcommand{\Po}{\mathsf{PSP(open)}}
\newcommand{\MAc}{\mathsf{MA(countable)}}
\newcommand{\MAsigma}{\mathsf{MA(\sigma\textrm{-}centered)}}
\newcommand{\CH}{\mathsf{CH}}
\newcommand{\MP}{\mathsf{MP}}
\newcommand{\MPp}{\mathsf{MP^+}}
\newcommand{\CB}{\mathsf{CB}}
\newcommand{\CBP}{\mathsf{CBP}}
\newcommand{\CBPp}{\mathsf{CBP^+}}
\newcommand{\CDH}{\mathsf{CDH}}
\newcommand{\RCDH}{\mathsf{RCDH}}
\newcommand{\ZFC}{\mathsf{ZFC}}
\newcommand{\Lim}{\mathsf{Lim}}
\newcommand{\COF}{\mathsf{Cof}}
\newcommand{\FIN}{\mathsf{Fin}}
\newcommand{\Gd}{\mathsf{G_\delta}}
\newcommand{\Aa}{\mathcal{A}}
\newcommand{\BB}{\mathcal{B}}
\newcommand{\Ss}{\mathcal{S}}
\newcommand{\CC}{\mathcal{C}}
\newcommand{\DD}{\mathcal{D}}

\newcommand{\FF}{\mathcal{F}}
\newcommand{\GG}{\mathcal{G}}

\newcommand{\XX}{\mathcal{X}}

\newcommand{\II}{\mathcal{I}}

\newcommand{\UU}{\mathcal{U}}

\newcommand{\PP}{\mathcal{P}}

\newcommand{\PPP}{\mathbb{P}}

\newcommand{\QQQ}{\mathbb{Q}}
\newcommand{\cccc}{\mathfrak{c}}

\newcommand{\dddd}{\mathfrak{d}}
\newcommand{\rrrr}{\mathfrak{r}}
\newcommand{\uuuu}{\mathfrak{u}}
\newcommand{\gggg}{\mathfrak{g}}

\begin{document}

\title{Seven characterizations of non-meager P-filters}

\author{Kenneth Kunen}
\address{Department of Mathematics
\newline\indent University of Wisconsin
\newline\indent 480 Lincoln Drive
\newline\indent Madison, WI 53706, USA}
\email{kunen@math.wisc.edu}
\urladdr{http://www.math.wisc.edu/\~{}kunen/}

\author{Andrea Medini}
\address{Kurt G\"odel Research Center for Mathematical Logic
\newline\indent University of Vienna
\newline\indent W\"ahringer Stra{\ss}e 25
\newline\indent A-1090 Wien, Austria}
\email{andrea.medini@univie.ac.at}
\urladdr{http://www.logic.univie.ac.at/\~{}medinia2/}

\author{Lyubomyr Zdomskyy}
\address{Kurt G\"odel Research Center for Mathematical Logic
\newline\indent University of Vienna
\newline\indent W\"ahringer Stra{\ss}e 25
\newline\indent A-1090 Wien, Austria}
\email{lyubomyr.zdomskyy@univie.ac.at}
\urladdr{http://www.logic.univie.ac.at/\~{}lzdomsky/}


\thanks{The second-listed and third-listed authors were supported by the FWF grant I 1209-N25. The third-listed author would also like to thank the Austrian Academy of Sciences for its generous support through the APART Program.}

\date{October 3, 2014}

\begin{abstract}
We give several topological/combinatorial conditions that, for a filter on
$\omega$, are equivalent to being a non-meager $\mathsf{P}$-filter. In
particular, we show that a filter is countable dense homogeneous if and only if it is a non-meager $\mathsf{P}$-filter. Here, we identify a filter with a subspace of $2^\omega$ through characteristic functions. Along the way, we generalize to non-meager $\mathsf{P}$-filters a result of Miller about $\mathsf{P}$-points, and we employ and give a new proof of results of Marciszewski. We also employ a theorem of Hern\'andez-Guti\'errez and Hru\v{s}\'ak, and answer two questions that they posed. Our result also resolves several issues raised by Medini and Milovich, and proves false one ``theorem'' of theirs. Furthermore, we show that the statement ``Every non-meager filter contains a non-meager $\mathsf{P}$-subfilter'' is independent of $\ZFC$ (more precisely, it is a consequence of $\uuuu<\gggg$ and its negation is a consequence of $\Diamond$). It follows from results of Hru\v{s}\'ak and van Mill that, under $\uuuu<\gggg$, a filter has less than $\cccc$ types of countable dense subsets if and only if it is a non-meager $\mathsf{P}$-filter. In particular, under $\uuuu<\gggg$, there exists an ultrafilter with $\cccc$ types of countable dense subsets. We also show that such an ultrafilter exists under $\MAc$.
\end{abstract}

\maketitle

By \emph{filter} we mean filter on $\omega$, unless we explicitly say otherwise. Furthermore, we assume that $\COF\subseteq\FF$ for every filter $\FF$, where $\COF=\{x\subseteq\omega:|\omega\setminus x|<\omega\}$. We identify every filter with a subspace of $2^\omega$ by identifying every subset of $\omega$ with its characteristic function. In particular, we say that a filter is \emph{non-meager} if it is non-meager as a subset of $2^\omega$. It is well-known that ultrafilters are non-meager (see for example \cite[Section 2]{medinimilovich}). A filter $\FF$ is a \emph{$\mathsf{P}$-filter} if for every countable $\XX\subseteq\FF$ there exists $z\in\FF$ such that $|z\setminus x|<\omega$ for every $x\in\XX$. An ultrafilter that is a $\mathsf{P}$-filter is called a \emph{$\mathsf{P}$-point}. Ketonen showed that $\mathsf{P}$-points (hence non-meager $\mathsf{P}$-filters) exist under $\dddd=\cccc$ (see \cite[Theorem 9.25]{blass}), while by a theorem of Shelah it is consistent that there are no $\mathsf{P}$-points (see \cite[Theorem 4.4.7]{bartoszynskijudah}). On the other hand, the following is a long-standing open problem\footnote{\,It is known, however, that the statement ``There are no non-meager $\mathsf{P}$-filters'' has large cardinal strength (see \cite[Corollary 4.11]{justmathiasprikrysimon} or \cite[Corollary 4.4.15]{bartoszynskijudah}).} (see \cite[Question 0.1]{justmathiasprikrysimon} or \cite[Section 4.4.C]{bartoszynskijudah}).
\begin{question}[Just, Mathias, Prikry, Simon]\label{nonmeagerPquestion}
Is it possible to prove in $\ZFC$ that there exists a non-meager $\mathsf{P}$-filter?
\end{question}
\noindent Even though Question \ref{nonmeagerPquestion} was not the original motivation for our research, we hope that the results obtained here might shed some light on it.

By \emph{space} we mean separable metrizable topological space. By
\emph{countable} we mean at most countable. Recall that a space $X$ is \emph{completely Baire}\footnote{\,Some authors use ``hereditarily Baire'' or even ``hereditary Baire'' instead of ``completely Baire'' (see for example \cite{marciszewski}).} (briefly, $\CB$) if every closed subspace of $X$ is a Baire space. Recall that a space $X$ is \emph{countable dense homogeneous} (briefly, $\CDH$) if for every pair $(D,E)$ of countable dense subsets of $X$ there exists a homeomorphism $h:X\longrightarrow X$ such that $h[D]=E$. See \cite[Sections 14-16]{arkhangelskiivanmill} for a nice introduction to countable dense homogeneity.

Identifying filters with subsets of $2^\omega$ is certainly not a new approach to the study of filters. For example, much is known about the delicate interplay between category and measure in this context (see \cite[Chapter 4]{bartoszynskijudah}). However, in this article, we will focus on properties of a different flavor, such as being $\CB$ or being $\CDH$ in the subspace topology, and investigate their relationship with the combinatorial property of being a non-meager\footnote{\,That non-meager filters can be characterized combinatorially is a well-known result of Talagrand (see Theorem \ref{talagrandcriterion}).} $\mathsf{P}$-filter. In fact, as one might suspect from the title, our main result (Theorem \ref{main}) shows that being $\CB$, being $\CDH$, and several other properties (that are not equivalent for arbitrary spaces) become equivalent (to being a non-meager $\mathsf{P}$-filter) when the spaces under consideration are filters. The following characterization was already known (see \cite[Theorem 1.2]{marciszewski}), and the proof of the right-to-left direction is used in the proof of Theorem \ref{main}.
\begin{theorem}[Marciszewski]\label{Pequalscb}
Let $\FF$ be a filter. Then $\FF$ is a non-meager $\mathsf{P}$-filter if and only if $\FF$ is $\CB$.
\end{theorem}
 
Apart from the ``naturalistic'' interest of this line of research, it is worth noting that non-meager filters can be a fruitful source of counterexamples in general topology. In fact, a non-meager filter $\FF$ is automatically a Baire topological group that is neither analytic nor coanalytic (by the arguments in \cite[Section 2]{medinimilovich}), while, by constructing $\FF$ carefully, one can ensure that it possesses further topological properties. For example, Medini and Milovich showed that under $\MAc$ there exists a $\CDH$ ultrafilter (see \cite[Theorem 21]{medinimilovich}), and used the same methods to answer a question of Hru\v{s}\'ak and Zamora Avil\'es.\footnote{\,More precisely, they showed that under $\MAc$ there exists an ultrafilter $\UU$ such that $\UU^\omega$ is $\CDH$ (see \cite[Theorem 24]{medinimilovich}). This gives a (consistent) example of a non-Polish subspace $X$ of $2^\omega$ such that $X^\omega$ is $\CDH$, which is what \cite[Question 3.2]{hrusakzamoraaviles} asks for.} As another example, Repov\v{s}, Zdomskyy and Zhang recently constructed a non-meager filter $\FF$ that is not $\CDH$ (see \cite[Theorem 1]{repovszdomskyyzhang}), thus strenghtening a result of van Mill.\footnote{\,Both $\FF$ and the example $X$ of van Mill (see \cite{vanmills}) are strongly locally homogeneous Baire spaces that are not $\CDH$. On the other hand, $\FF$ is a topological group, while $X$ is merely homogeneous (see the discussion in \cite[page 1323]{medinimilovich}).} Both results can now be viewed as corollaries of Theorem \ref{main}.

The following result (see \cite[Theorem 1.6]{hernandezgutierrezhrusak}) improves on the example of Medini and Milovich mentioned above, and its proof is used in the proof of Theorem \ref{main}.
\begin{theorem}[Hern\'andez-Guti\'errez, Hru\v{s}\'ak]\label{Pimpliescdh}
If $\FF$ is a non-meager $\mathsf{P}$-filter then $\FF$ is $\CDH$.
\end{theorem}
\noindent The article \cite{medinimilovich} also contains the claim that, under $\MAc$, there exists a $\CDH$ ultrafilter that is not a $\mathsf{P}$-point (see \cite[Theorem 41]{medinimilovich}). Unfortunately, the proof is wrong: \cite[Lemma 42]{medinimilovich} is correct, but it is easy to realize that a stronger lemma is needed. In fact, as Theorem \ref{main} shows, the claim itself is false.

An important step towards our main result is achieved in Section 2, where we generalize to non-meager $\mathsf{P}$-filters a result of Miller about $\mathsf{P}$-points. Inspired by his result, we give an explicit definition of a topological property (the Miller property) which seems to be of independent interest. This will allow us to give a new, more systematic proof of the left-to-right direction of Theorem \ref{Pequalscb}. Furthermore, the Miller property will be the key to proving that every $\CDH$ filter must be a non-meager $\mathsf{P}$-filter.

In Section 3, we give the seven characterizations promised in the title and use them to answer several questions from the literature. Inspired by the classical Cantor-Bendixson derivative, we also introduce a topological property (the Cantor-Bendixson property) which seems to be of independent interest. By Theorem \ref{main}, all the properties that we mentioned so far (and stronger versions of some of them) are equivalent for filters. It would be rather silly if some of these properties were actually equivalent for arbitrary spaces. By giving suitable counterexamples, we show that this is not the case.

At this point, it seems natural to investigate whether we can say more about the number of types of countable dense subsets of a filter. Recall that the \emph{type} of a countable dense subset $D$ of a space $X$ is $\{h[D]:h\text{ is a homeomorphism of }X\}$.
In particular, a space is $\CDH$ if and only if it has exactly $1$ type of countable dense subsets. Also notice that the maximum possible number of types of countable dense subsets of a space is $\cccc$. See \cite{hrusakvanmill} for more on this topic.

In Section 6, we show that it is consistent that every filter has either $1$ or $\cccc$ types of countable dense sets. More precisely, under the assumption $\uuuu<\gggg$, a filter has less than $\cccc$ types of countable dense subsets if and only if it is a non-meager $\mathsf{P}$-filter (see Theorem \ref{anotherequivalent}). To achieve this, we employ techniques of Hru\v{s}\'ak and van Mill (see Section 4) plus the fact that, under $\uuuu<\gggg$, every non-meager filter has a non-meager $\mathsf{P}$-subfilter (see Section 5). In Section 8, assuming $\Diamond$, we construct an ultrafilter with no non-meager $\mathsf{P}$-subfilters, thus showing that the statement ``Every non-meager filter contains a non-meager $\mathsf{P}$-subfilter'' is independent of $\ZFC$.

While the existence in $\ZFC$ of an ultrafilter that is not $\CDH$ follows easily from Theorem \ref{main} (see Corollary \ref{notcdhzfc}), we do not know whether it is possible to construct in $\ZFC$ an ultrafilter (or a non-meager filter) with $\cccc$ types of countable dense subsets (see Question \ref{ctypesquestion}). It follows from our consistent characterization that such an ultrafilter exists under $\uuuu<\gggg$ (see Corollary \ref{ctypesunderulessg}). In Section 7, we show that such an ultrafilter also exists under $\MAc$ (see Corollary \ref{ctypesundermac}).

\section{More preliminaries}

Our reference for general topology is \cite{vanmilli}. For notions related to cardinal invariants, we refer to \cite{blass}. For all other set-theoretic notions, we refer to \cite{kunen}.

Recall that a space is \emph{crowded} if it is non-empty and it has no isolated points. We write $X\approx Y$ to mean that the spaces $X$ and $Y$ are homeomorphic. Given spaces $X$ and $Z$, a \emph{copy} of $Z$ in $X$ is a subspace $Y$ of $X$ such that $Y\approx Z$. We say that a subspace $X$ of $2^\omega$ is \emph{relatively countable dense homogeneous} (briefly, $\RCDH$) if for every pair $(D,E)$ of countable dense subsets of $X$ there exists a homeomorphism $h:2^\omega\longrightarrow 2^\omega$ such that $h[X]=X$ and $h[D]=E$. We will need the following classical result (see \cite[Corollary 1.9.13]{vanmilli}) on several occasions.
\begin{theorem}[Hurewicz]\label{hurewicz}
A space is $\CB$ if and only if it does not contain any closed copy of $\QQQ$.
\end{theorem}

We denote by $\PP(\omega)$ the collection of all subsets of $\omega$. Whenever $\XX\subseteq\PP(\omega)$, we freely identify $\XX$ with the subspace $\XX\subseteq 2^\omega$ consisting of the characteristic functions of the elements of $\XX$. Let $\FIN=\{x\subseteq\omega:|x|<\omega\}$. Given $z\subseteq\omega$, let $z\!\uparrow\,\,=\{x\subseteq\omega:z\subseteq x\}$ and $z\!\downarrow\,\,=\{x\subseteq\omega:x\subseteq z\}$.

Given $x,y\subseteq\omega$, we will write $x\subseteq^\ast y$ to mean $|x\setminus y|<\omega$. Given $\XX\subseteq\PP(\omega)$, we will say that $z\subseteq\omega$ is a \emph{pseudointersection} of $\XX$ if $z$ is infinite and $z\subseteq^\ast x$ for all $x\in\XX$. In particular, a filter $\FF$ is a $\mathsf{P}$-filter if and only if every countable $\XX\subseteq\FF$ has a pseudointersection in $\FF$. A \emph{subfilter} of a filter $\FF$ is a filter $\GG$ such that $\GG\subseteq\FF$. A subfilter that is a $\mathsf{P}$-filter is called a \emph{$\mathsf{P}$-subfilter}.

Recall that $\XX\subseteq\PP(\omega)$ has the \emph{finite intersection property} if $\bigcap F$ is infinite for every non-empty $F\in[\XX]^{<\omega}$. Given $\XX\subseteq\PP(\omega)$ with the finite intersection property, the \emph{filter generated} by $\XX$ is
$$
\FF=\COF\cup\{x\subseteq\omega:\bigcap F\subseteq^\ast x\text{ for some non-empty }F\in[\XX]^{<\omega}\}.
$$
It is easy to check that $\FF$ is the smallest filter such that $\XX\subseteq\FF$. Given $x\subseteq\omega$, define $x^0=\omega\setminus x$ and $x^1=x$. Recall that $\Aa\subseteq\PP(\omega)$ is an \emph{independent family} if $\{x^{\nu(x)}:x\in\Aa\}$ has the finite intersection property for every $\nu:\Aa\longrightarrow 2$.

The following well-known characterization of non-meager filters (see \cite[Proposition 9.4]{blass}) originally appeared as part of \cite[Th\'eor\`eme 21]{talagrand}, and it will prove very useful for our purposes.
\begin{theorem}[Talagrand]\label{talagrandcriterion}
For a filter $\FF$, the following conditions are equivalent.
\begin{itemize}
\item $\FF$ is non-meager.
\item For every partition $\Pi$ of $\omega$ into finite sets there exists $x\in\FF$ such that $x\cap I=\varnothing$ for infinitely many $I\in\Pi$.
\end{itemize}
\end{theorem}

\section{Strengthening a result of Miller}

Miller showed that $\mathsf{P}$-points are preserved in rational perfect set forcing extensions (see \cite[Theorem 3.1]{millerf}), and remarked that his proof can be adapted to obtain Theorem \ref{millerPpoint}.

\begin{definition}
A space $X$ has the \emph{Miller property} (briefly, $\MP$) if for every countable crowded subspace $Q$ of $X$ there exists a copy $K$ of $2^\omega$ in $X$ such that $K\cap Q$ is crowded. A subspace $X$ of $2^\omega$ has the \emph{strong Miller property} (briefly, $\MPp$) if for every countable crowded subspace $Q$ of $X$ there exists a copy $K$ of $2^\omega$ in $X$ such that $K\cap Q$ is crowded and $K\subseteq z\!\uparrow$ for some $z\in X$.
\end{definition}
\noindent Notice that the $\MPp$ implies the $\MP$ for every subspace of $2^\omega$. See the next section for a counterexample to the reverse implication.

\begin{theorem}[Miller]\label{millerPpoint}
Every $\mathsf{P}$-point has the $\MPp$.
\end{theorem}

Next, we generalize Miller's result to non-meager $\mathsf{P}$-filters (see Corollary \ref{generalizedmillercor}) by suitably modifying his proof. This will be a crucial ingredient in the proof of Theorem \ref{main}. In fact, it will allow us to give a new, more systematic proof of the left-to-right direction of Theorem \ref{Pequalscb}, and to show that having the $\MP$ is actually equivalent to being a non-meager $\mathsf{P}$-filter. Finally, using this characterization, we will be able to prove that a $\CDH$ filter must be a non-meager $\mathsf{P}$-filter.

\begin{lemma}\label{generalizedmiller}
Let $\FF$ be a non-meager filter. Let $Q$ be a countable crowded
subspace of $\FF$ such that $Q$ has a pseudointersection in $\FF$. Then there exists a crowded $Q'\subseteq Q$ such that $Q'\subseteq z\!\uparrow$ for some $z\in\FF$.
\end{lemma}
\begin{proof}
Fix $x\in\FF$ such that $x\subseteq^\ast q$ for all $q\in Q$. Let
$$
\Ss=\{\varnothing\}\cup\{s\in 2^{<\omega}:|s|\geq 1\text{ and
}s(|s|-1)=0\}.
$$
\noindent We will also need a bookkeeping function
$f:\omega\longrightarrow\omega$ such that the following conditions are
satisfied.
\begin{itemize}
\item $f(n)<n$ for every $n\geq 1$. 
\item $f^{-1}(m)$ is infinite for every $m\in\omega$.
\end{itemize}
Constructing such a function is an easy exercise, left to the reader.

We will recursively choose natural numbers $k_0<k_1<\cdots$ and $q_s\in Q$ for
$s\in\Ss$. By induction, we will make sure that the following conditions are
satisfied. Define $\ell_s\in\omega$ for every $s\in 2^{<\omega}$ so that
$\{t\in\Ss:t\subsetneq s\}=\{t^s_i:i<\ell_s\}$, where
$\varnothing=t^s_0\subsetneq\cdots\subsetneq t^s_{\ell_s-1}$. Also set
$q^s_i=q_{t^s_i}$ for every $i<\ell_s$. Notice that if $s'\in 2^{<\omega}$ and
$s'\supseteq s$ then $t^s_i=t^{s'}_i$ and $q^s_i=q^{s'}_i$ for every $i<\ell_s$.
\begin{enumerate}
\item\label{distinctalongbranch} $t\subsetneq s$ implies $q_t\neq q_s$ for all
$t,s\in\Ss$.
\item\label{trulyincluded} $x\setminus
k_n\subseteq\bigcap\{q_s:s\in\Ss\text{ and }|s|\leq n\}$ for all
$n\in\omega$.
\item\label{crowded} $q_s\upharpoonright
k_{|s|-1}=q^s_{f(\ell_s)}\upharpoonright k_{|s|-1}$ for every $s\in\Ss$ such
that $|s|\geq 1$.
\end{enumerate}

Start by letting $q_\varnothing$ be any element of $Q$. Assume without loss of
generality that $x\subseteq q_\varnothing$, and let $k_0=0$. Now fix $n\geq 1$.
Assume that $q_t$ has been constructed for every $t\in\Ss$ such that $|t|<n$,
and that $k_i$ has been constructed for every $i<n$. Fix $s\in\Ss$ such that
$|s|=n$. Define $q_s$ to be any element of
$$
(Q\cap [q^s_{f(\ell_s)}\upharpoonright k_{n-1}])\setminus\{q^s_i:i<\ell_s\}.
$$
Now simply choose $k_n>k_{n-1}$ big enough so that condition
$(\ref{trulyincluded})$ is satisfied.

Since $\FF$ is a non-meager filter, applying Theorem \ref{talagrandcriterion} yields a function $\phi:\omega\longrightarrow 2$
such that $\phi^{-1}(0)$ is infinite and $w\in\FF$ such that
$$
w\cap\bigcup\{[k_n,k_{n+1}):n\in\phi^{-1}(0)\}=\varnothing.
$$
Let $z=x\cap w$ and $Q'=\{q_s:s\subseteq\phi\text{ and }s\in\Ss\}$.

First we will show that $Q'$ is crowded. So let $q\in Q'$ and fix
$\ell\in\omega$. We will find $q'\in Q'$ such that $q'\neq q$ and
$q'\upharpoonright\ell=q\upharpoonright\ell$. Let $s\subseteq\phi$ and
$m<\ell_s$ be such that $q=q^s_m$. Notice that $m=f(\ell_{s'})$ for infinitely
many $s'\in\Ss$ such that $s'\subseteq\phi$, therefore it is possible to
choose one with $n=|s'|>|s|$ big enough so that $k_{n-1}\geq\ell$. Let
$q'=q_{s'}$. Condition $(\ref{distinctalongbranch})$ implies that $q'\neq q$.
Condition $(\ref{crowded})$ implies that $q'\upharpoonright
k_{n-1}=q^{s'}_{f(\ell_{s'})}\upharpoonright k_{n-1}$, which is sufficient
because $q^{s'}_{f(\ell_{s'})}=q^s_m$.

In conclusion, we will use induction on $n$ to show that $z\subseteq
q_{\phi\upharpoonright n}$ for every $n\in\omega$ such that $\phi\upharpoonright
n\in\Ss$. The claim is clear for $n=0$ by the choice of $x$. Now assume that
$n\geq 1$ and $s=\phi\upharpoonright n\in\Ss$. Let $k\in z$. We will show that
$k\in q_s$ by considering the following cases.
\begin{itemize}
\item $k\in [0,k_{n-1})$.
\item $k\in [k_{n-1},k_n)$.
\item $k\in [k_n,\infty)$.
\end{itemize}
By condition $(\ref{crowded})$, there exists an $m<\ell_s$ such that
$q_s\upharpoonright k_{n-1}=q^s_m\upharpoonright k_{n-1}$. Since $q^s_m=q_t$ for
some $t\in\Ss$ such that $t\subsetneq s$, the inductive hypothesis guarantees
that $z\subseteq q^s_m$. So $k\in q_s$ in the first case. Notice that
$\phi(n-1)=s(n-1)=0$, so $[k_{n-1},k_n)\cap z=\varnothing$. This shows that the
second case never happens. Finally, condition $(\ref{trulyincluded})$ implies
that $k\in q_s$ in the third case.
\end{proof}
\begin{corollary}\label{generalizedmillercor}
Every non-meager $\mathsf{P}$-filter has the $\MPp$.
\end{corollary}

\section{The main result}

This section contains our main result, which gives the seven characterizations promised in the title (see Theorem \ref{main}). The proof of the implication $(\ref{completelybaire})\rightarrow(\ref{nonmeagerPfilter})$ is due to Marciszewski (see \cite[Lemma 2.1]{marciszewski}), and it is included for completeness.

\begin{definition}
A space $X$ has the \emph{Cantor-Bendixson property} (briefly, $\CBP$) if every closed subspace of $X$ is either scattered or it contains a copy of $2^\omega$. A subspace $X$ of $2^\omega$ has the \emph{strong Cantor-Bendixson property} (briefly, $\CBPp$) if every closed subspace of $X$ is either scattered or it contains a copy $K$ of $2^\omega$ such that $K\subseteq z\!\uparrow$ for some $z\in X$.
\end{definition}
\noindent Notice that the $\CBPp$ implies the $\CBP$ for every subspace of $2^\omega$. Furthermore, one can easily check that the $\MP$ implies the $\CBP$ for every space, and that the $\MPp$ implies the $\CBPp$ for every subspace of $2^\omega$. Finally, using Theorem \ref{hurewicz}, one can show that every space with the $\CBP$ is $\CB$.

The above definition is of course motivated by the classical Cantor-Bendixson derivative. Notice that the $\CBPp$ is to the $\CBP$ what the $\MPp$ is to the $\MP$. In both cases, one version of the property is purely topological, while the other requires that the copy $K$ of $2^\omega$ can be bounded (in the ordering given by reverse-inclusion) by an element of the space.\footnote{\,This is partly inspired by \cite{millers}, where Miller studies the relation between property $\Ps$ and preservation by Sacks forcing for ultrafilters. Recall that a subset $X$ of $2^\omega$ has \emph{property $\Ps$} (or is \emph{Marczewski measurable}) if every copy $K$ of $2^\omega$ in $2^\omega$ contains a copy $K'$ of $2^\omega$ such that $K'\subseteq X$ or $K'\subseteq 2^\omega\setminus X$. We say that an ultrafilter $\UU$ has \emph{property $\Psp$} if every copy $K$ of $2^\omega$ in $2^\omega$ contains a copy $K'$ of $2^\omega$ such that $K'\subseteq z\!\uparrow$ or $K'\subseteq (\omega\setminus z)\!\downarrow$ for some $z\in\UU$. This notion is due to Miller (even though he did not give it a name), who obtained the following results (see \cite[Theorem 1]{millers} and \cite[Theorem 3]{millers} respectively).
\begin{itemize}
\item An ultrafilter has property $\Psp$ if and only if it is preserved by Sacks forcing.
\item Assume $\MAc$. Then there exists an ultrafilter that has property $\Ps$ but not property $\Psp$.
\end{itemize}
The second result seems particularly interesting to us, because it exhibits a property of ultrafilters that is not equivalent to its strong version.} Furthermore, it is easy to realize that the strong versions of these properties are only of interest under some additional combinatorial assumption on $X$ (such as being a filter). In fact, given a coinfinite $z\subseteq\omega$ and a non-empty zero-dimensional space $Z$, one can always find a subspace $X$ of $2^\omega$ such that $X\approx Z$ and $z\in X\subseteq z\!\uparrow$ (since $z\!\uparrow\,\approx 2^\omega$ and $2^\omega$ is homogeneous). In particular, every zero-dimensional space with the $\MP$ (respectively, the $\CBP$) is homeomorphic to a subspace of $2^\omega$ with the $\MPp$ (respectively, the $\CBPp$).

As we already mentioned, some obvious relationships hold among the properties that we considered so far. Next, we will show that the implications in the following diagram (and their obvious consequences) are the only ones that hold in $\ZFC$ for arbitrary subspaces of $2^\omega$.

\smallskip

\begin{center}
$
\xymatrix{
& \MP \ar@/^/[rd]& & & &\\
  \MPp \ar@/^/[ru] \ar@/_/[rd] & & \CBP \ar@{->}[r] &\CB& &\RCDH \ar@{->}[r]&\CDH\\
& \CBPp \ar@/_/[ru]& & & & 
}
$
\end{center}

\smallskip

For an example (based on a result of Brendle) of a subspace of $2^\omega$ that has the $\CBP$ but not the $\MP$, see \cite[Proposition 3.2]{medinizdomskyy}. For an example of a $\CB$ subspace of $2^\omega$ without the $\CBP$, see \cite[Proposition 3.3]{medinizdomskyy}. 

To see that the $\MP$ does not imply the $\MPp$ and that the $\CBP$ does not imply the $\CBPp$, a single example will suffice. Let $\Aa$ be an independent family that is homeomorphic to $2^\omega$ (see \cite[Lemma 7]{medinimilovich}). Since $\Aa$ is compact, it is clear that $\Aa$ has the $\MP$. Notice that $\Aa\cap (z\!\uparrow)=\{z\}$ for every $z\in\Aa$ because $\Aa$ is an independent family. Since $\Aa$ is a non-scattered closed subspace of itself, it follows that $\Aa$ does not have the $\CBPp$.

To see that the $\CBPp$ does not imply the $\MPp$, let $X$ be a subspace of $2^\omega$ that has the $\CBP$ but not the $\MP$. As we mentioned above, we can assume without loss of generality that $z\in X\subseteq z\!\uparrow$ for some coinfinite $z\subseteq\omega$ (for example $z=\varnothing$). It is trivial to check that $X$ has the desired properties.

For an example, under $\MAsigma$, of a $\CDH$ subspace of $2^\omega$ that is not $\RCDH$, see \cite[Corollary 10]{medinivanmillzdomskyy}. Furthermore, an $\RCDH$ subspace of $2^\omega$ need not be $\CB$. In fact, Hern\'andez-Guti\'errez, Hru\v{s}\'ak and van Mill recently gave $\ZFC$ examples of meager $\RCDH$ dense subspaces of $2^\omega$ (see \cite[Theorem 4.1]{hernandezgutierrezhrusakvanmill}).

Finally, to see that a subspace of $2^\omega$ with the $\MPp$ need not be $\CDH$, consider $X=\{\omega\setminus n:n\in\omega\}\cup\{\varnothing\}$. Since $X\approx\omega +1$ and $\varnothing\in X$, it is clear that $X$ has the desired properties.

\begin{theorem}\label{main}
For a filter $\FF$, the following conditions are equivalent.
\begin{enumerate}
\item\label{nonmeagerPfilter} $\FF$ is a non-meager $\mathsf{P}$-filter.
\item\label{millerpropertyplus} $\FF$ has the $\MPp$.
\item\label{millerproperty} $\FF$ has the $\MP$.
\item\label{cantorbendixsonplus} $\FF$ has the $\CBPp$.
\item\label{cantorbendixson} $\FF$ has the $\CBP$.
\item\label{completelybaire} $\FF$ is $\CB$.
\item\label{cdhp} $\FF$ is $\RCDH$.
\item\label{cdh} $\FF$ is $\CDH$.
\end{enumerate}
\end{theorem}
\begin{proof}
First we will show that the first six properties are equivalent.
The implication $(\ref{nonmeagerPfilter})\rightarrow(\ref{millerpropertyplus})$ is the content of Corollary \ref{generalizedmillercor}. Given the discussion above, it will be enough to prove the implication $(\ref{completelybaire})\rightarrow(\ref{nonmeagerPfilter})$. Assume that $\FF$ is either meager or not a $\mathsf{P}$-filter. If $\FF$ is meager then $\FF$ is not even Baire. So assume that $\FF$ is not a $\mathsf{P}$-filter. Then there exists a sequence $x_0\supseteq x_1\supseteq\cdots$ consisting of elements
of $\FF$ with no pseudointersection in $\FF$. Without loss of generality, we can
assume that $c_n=x_n\setminus x_{n+1}$ is infinite for each $n$
and $\bigcup_{n\in\omega}c_n=\omega$. Hence, we can also assume that $\FF$
is a filter on $\omega\times\omega$ and $c_n=\{n\}\times\omega$ for each $n$. It
is clear that the following conditions hold.
\begin{itemize}
\item For all $x\in\FF$ there exists $n\in\omega$ such that $x\cap c_n$ is
infinite.
\item $\bigcup_{m\geq n}c_m\in\FF$ for every $n\in\omega$.
\end{itemize}
\noindent Consider the set $Q\subseteq 2^{\omega\times\omega}$ consisting of all $x\in\FF$ that satisfy the following requirements.
\begin{itemize}
\item If $(i,j)\in x$ and $k\geq i$ then $(k,j)\in x$.
\item If $(i,j)\in x$ and $k\leq j$ then $(i,k)\in x$.
\end{itemize}
It is not hard to see that $Q$ is a countable crowded closed subspace of $\FF$. Therefore $\FF$ is not $\CB$.

We will finish the proof by showing that $(\ref{nonmeagerPfilter})\rightarrow
(\ref{cdhp})\rightarrow (\ref{cdh})\rightarrow (\ref{millerproperty})$. The
implication $(\ref{nonmeagerPfilter})\rightarrow(\ref{cdhp})$ follows from the proof of \cite[Theorem 1.6]{hernandezgutierrezhrusak}. The implication $(\ref{cdhp})\rightarrow(\ref{cdh})$ is obvious. In order to show that $(\ref{cdh})\rightarrow (\ref{millerproperty})$,
assume that $\FF$ is $\CDH$. Fix a countable crowded subspace $Q$ of $\FF$. Extend $Q$ to a countable dense subset $D$ of
$\FF$ and let $E=\COF$. Since $\FF$ is $\CDH$, there exists a homeomorphism $h:\FF\longrightarrow\FF$ such that $h[D]=E$. Let $R=h[Q]$ and observe that $R$ is a
countable crowded subspace of $\FF$ with $\omega\in\FF$ as a pseudointersection. Also notice that $\FF$ must be non-meager by Corollary \ref{meagerfiltersmanytypes}. Therefore, by Lemma \ref{generalizedmiller}, there exists a crowded $R'\subseteq
R$ and $z\in\FF$ such that $R'\subseteq z\!\uparrow$. In particular,
$R'$ has compact closure in $\FF$. Let $Q'=h^{-1}[R']$. Since $h$ is a homeomorphism, it follows that $Q'\subseteq Q$ is crowded and has compact closure in $\FF$. It is clear that the closure $K$ of $Q'$ in $\FF$ is a copy of $2^\omega$ such that $K\cap Q$ is crowded.
\end{proof}

Clearly, Theorem \ref{main} implies that, for an ultrafilter, being a
$\mathsf{P}$-point is equivalent to being $\CDH$. Hence, the well-known fact that there exist ultrafilters that are not $\mathsf{P}$-points (simply apply Lemma \ref{nonP} with $\Aa=\varnothing$) immediately yields the following corollary. This answers \cite[Question 2]{medinimilovich} and simultaneously strengthens \cite[Theorem 15]{medinimilovich} (which gives, under $\MAc$, an ultrafilter that is not $\CDH$) and \cite[Theorem 1]{repovszdomskyyzhang} (which
gives a non-meager filter that is not $\CDH$).

\begin{corollary}\label{notcdhzfc}
There exists an ultrafilter that is not $\CDH$.
\end{corollary}

Furthermore, since Shelah showed that it is consistent that there
are no $\mathsf{P}$-points (see \cite[Theorem 4.4.7]{bartoszynskijudah}), it follows that it is consistent that there are no $\CDH$ ultrafilters. This answers \cite[Question 3]{medinimilovich}.

Similarly, Theorem \ref{main} implies that for an ultrafilter, being a
$\mathsf{P}$-point is equivalent to being $\CB$. This answers \cite[Question 10]{medinimilovich}. Therefore, as above, it is consistent that there are no $\CB$ ultrafilters. This answers \cite[Question 1]{medinimilovich}. However, the answer to both questions follows already from Theorem \ref{Pequalscb}, of which the authors of \cite{medinimilovich} were not aware. Also notice that Theorem \ref{main} answers \cite[Question 4]{medinimilovich} (which asks whether a $\CDH$ ultrafilter is necessarily $\CB$).

We also remark that Theorem \ref{main} and Corollary \ref{notcdhzfc} answer two questions of Hern\'andez-Guti\'errez and Hru\v{s}\'ak, and clarify a third. More precisely, the equivalence $(\ref{nonmeagerPfilter})\leftrightarrow (\ref{cdh})$ answers \cite[Question 1.8]{hernandezgutierrezhrusak} (which asks for a combinatorial
characterization of $\CDH$ filters), while Corollary \ref{notcdhzfc} answers the second part of \cite[Question 1.9]{hernandezgutierrezhrusak}. The first part of
\cite[Question 1.9]{hernandezgutierrezhrusak} asks whether the existence of a $\CDH$ filter can be proved in $\ZFC$. By Theorem \ref{main}, this is equivalent to Question \ref{nonmeagerPquestion}, which is a long-standing open problem (see the introduction).

Finally, it is natural to ask whether Theorem 10 can be improved. As we have seen, none of the equivalences can be extended to arbitrary subspaces of $2^\omega$. However, we do not know to what extent the combinatorial assumptions on $\FF$ can be relaxed. Recall that a collection $\FF\subseteq\PP(\omega)$ is a \emph{semifilter} if $\COF\subseteq\FF$ and $\FF$ is closed under supersets and finite modifications of its elements (see \cite{banakhzdomskyy}).
\begin{question}
Exactly which fragments of Theorem \ref{main} remain valid for semifilters?\footnote{\,It is not clear what the analogue of $\mathsf{P}$-filter should be for semifilters. The following is a plausible candidate. Define a semifilter $\FF$ to be a \emph{$\mathsf{P}$-semifilter} if every sequence $x_0\supseteq x_1\supseteq\cdots$ consisting of elements of $\FF$ has a pseudointersection in $\FF$. Notice that a filter is a $\mathsf{P}$-semifilter if and only if it is a $\mathsf{P}$-filter. Furthermore, it is easy to realize that the proof of the implication $(\ref{completelybaire})\rightarrow(\ref{nonmeagerPfilter})$ would generalize to this context.} For semifilters with the finite intersection property?
\end{question}

\section{How to obtain $\cccc$ types of countable dense subsets}

The main results of this section (Theorem \ref{hrusakvanmillctypes} and Theorem \ref{hrusakvanmillnotbaire}) give conditions under which a space is guaranteed to have $\cccc$ types of countable dense subsets, and are essentially due to Hru\v{s}\'ak and van Mill. In fact, even if they did not explicitly notice them, their proofs are taken almost verbatim from the proof of \cite[Theorem 4.5]{hrusakvanmill}. We decided to keep the weaker Proposition \ref{notcdh} as well, since its proof seems particularly transparent.

The first of these conditions involves spaces that contain a dense $\CB$ subspace. While it is easy to see that any space containing a dense Baire subspace must be Baire, a space containing a dense $\CB$ subspace need not be $\CB$. Consider for example the subspace $C\cup Q$ of $2^\omega\times 2^\omega$, where $C=2^\omega\times (2^\omega\setminus\{x\})$ and $Q$ is a countable dense subset of $2^\omega\times\{x\}$ for some fixed $x\in 2^\omega$.

Lemma \ref{nonhomeocountable} first appeared in \cite{mazurkiewicz}, then Brian, van Mill and Suabedissen gave a new proof (see \cite[Lemma 14]{brianvanmillsuabedissen} or \cite[Lemma 4.3]{hrusakvanmill}).
 
\begin{lemma}[Mazurkiewicz, Sierpi\'{n}ski]\label{nonhomeocountable}
There exists a family $\CC$ of size $\cccc$ consisting of pairwise
non-homeomorphic countable spaces.
\end{lemma}

\begin{proposition}\label{notcdh}
Assume that $X$ is not $\CB$ but has a dense $\CB$ subspace $C$. Then $X$ is not $\CDH$.
\end{proposition}
\begin{proof}
Let $D$ be a countable dense subset of $C$. By Theorem \ref{hurewicz}, we can fix a closed copy $Q$ of $\QQQ$ in $X$. Now extend $Q$ to a countable dense subset $E$ of $X$. Clearly, there is no homeomorphism $h:X\longrightarrow X$ such that $h[D]=E$.
\end{proof}

\begin{theorem}\label{hrusakvanmillctypes}
Assume that $X$ is not $\CB$ but has a dense $\CB$ subspace $C$. Then $X$ has $\cccc$ types of countable dense subsets.
\end{theorem}
\begin{proof}
By Theorem \ref{hurewicz}, we can fix a closed copy $Q$ of $\QQQ$ in $X$. Since $X$ is a Baire space (because it has a dense Baire subspace), $Q$ must be nowhere dense. Therefore, it is easy to obtain a countable dense subset $D$ of $X$ such that $D\subseteq C$ and $D\cap Q=\varnothing$.

By Lemma \ref{nonhomeocountable}, there exists a family $\CC$ of size $\cccc$ consisting of pairwise non-homeomorphic countable spaces. Since $Q\approx\QQQ\approx\QQQ^2$, we can also assume that every member of $\CC$ is a nowhere dense subspace of $Q$. For every $A\in\CC$, define $D_A=(Q\setminus
A)\cup D$. We claim that $D_A$ and $D_B$ are countable dense subsets of a different type whenever $A,B\in\CC$ and $A\neq B$.

Assume that $A,B\in\CC$ are such that there exists a homeomorphism $h:X\longrightarrow X$ such that $h[D_A]=D_B$. We will show that $A=B$. Assume, in order to get a contradiction, that $h(x)\notin Q$ for some $x\in Q\setminus A$. Since $D$ is a neighborhood of $h(x)$ in $D_B$, by continuity there exists a
neighborhood $U$ of $x$ in $Q\setminus A$ such that $h[U]\subseteq D$. Notice that $U\setminus\cl(A)$ is a non-empty open subset of $Q$ because $A$ is nowhere dense in $Q$. So there exists a non-empty open subset $V$ of $Q$ such that $V\subseteq\cl(V)\subseteq U\setminus\cl(A)\subseteq Q\setminus A$. It follows that $\cl(V)$ is a
copy of $\QQQ$ that is closed in $X$. But then $h[\cl(V)]\subseteq
D\subseteq C$ is also a copy of $\QQQ$ that is closed in $X$, which
contradicts the fact that $C$ is $\CB$.

So $h[Q\setminus A]\subseteq Q$. Since $h[D_A]=D_B$, we must have $h[Q\setminus A]\subseteq Q\setminus B$. A similar reasoning yields $h^{-1}[Q\setminus B]\subseteq Q\setminus A$. Therefore $h[Q\setminus A]= Q\setminus B$. Notice that $\cl(Q\setminus A)=\cl(Q\setminus B)=Q$. Since $h$ is a homeomorphism, it follows that $h[Q]=Q$. Hence $h[A]=B$, which concludes the proof.
\end{proof}

We will say that a space $X$ has the \emph{perfect set property for open sets} (briefly, $\Po$) if every uncountable open subset of $X$ contains a copy of $2^\omega$. Lemma \ref{meagerdenseGdelta} first appeared as \cite[Lemma 3.2]{fitzpatrickzhou}.

\begin{lemma}[Fitzpatrick, Zhou]\label{meagerdenseGdelta}
Every meager space has a countable dense
$\Gd$ subset.
\end{lemma}

\begin{theorem}\label{hrusakvanmillnotbaire}
Assume that $X$ has the $\Po$ but is not a Baire space. Then $X$ has $\cccc$ types of countable dense subsets.
\end{theorem}
\begin{proof}
Write $X$ as the disjoint union $S\cup C$, where $S$ is scattered open and $C$ is crowded. Notice that $C$ has the $\Po$ because $S$ is countable. Since $C$ is invariant under every homeomorphism of $X$, if $C$ has $\cccc$ types of countable dense subsets then the same is true for $X$. Furthermore, using the fact that every meager open subset of $X$ is disjoint from $\cl(S)$, it is easy to check that $C$ is not Baire. Therefore, we can assume without loss of generality that $X$ is crowded.

First assume that some non-empty open subset of $X$ is countable. Then 
$$
V=\bigcup\{U:U\text{ is a countable open subset of }X\}
$$
is a non-empty countable open subset of $X$. Since $X$ is crowded, it follows that $V$ is crowded, hence $V\approx\QQQ$. As in the proof of Theorem \ref{hrusakvanmillctypes}, there exists a family $\CC$ of size $\cccc$ consisting of pairwise non-homeomorphic countable nowhere dense subspaces of $V$. Fix a countable dense subset $D$ of $X\setminus V$. For every $A\in\CC$, define $D_A=(V\setminus
A)\cup D$. Since $V$ is invariant under every homeomorphism of $X$, it is clear that $D_A$ and $D_B$ are countable dense subsets of a
different type whenever $A,B\in\CC$ and $A\neq B$.

Now assume that every non-empty open subset of $X$ is uncountable. Since $X$ is not Baire, there exists a non-empty open subset $U$ of $X$ such that $U$ is meager and $X\setminus\cl(U)$ is non-empty. Since $X$ has the $\Po$, there exists a copy $K$ of $2^\omega$ contained in $U$. Notice that $K$ is nowhere dense because $K$ is compact and $U$ is meager. Therefore, using the compactness of $K$, it is possible to construct a regular open subset $W$ of $X$ such that $W\subseteq U$ and $K\subseteq\cl(W)\setminus W$.

Since $W$ is meager, it contains a countable dense $\Gd$ subset $D$ by Lemma \ref{meagerdenseGdelta}. Fix an open base $\{U_n:n\in\omega\}$ for $X\setminus\cl(W)$. Since $X$ has the $\Po$, each $U_n$ contains a copy $K_n$ of $2^\omega$. Fix a countable dense subset $E_n$ of each $K_n$. Let $E=\bigcup_{n\in\omega}E_n$.
Notice that $O\cap E$ is not a $\Gd$ subset of $O$ for any non-empty open subset $O$ of $X\setminus\cl(W)$, otherwise $K_n\cap E\supseteq E_n$ would be a countable dense $\Gd$ subset of $K_n\approx 2^\omega$ for some $n$.

By Lemma \ref{nonhomeocountable}, there exists a family $\CC$ of size $\cccc$ consisting of pairwise non-homeomorphic countable subspaces of $K$. For every $A\in\CC$, define $D_A=D\cup A\cup E$. We claim that $D_A$ and $D_B$ are countable dense subsets of a
different type whenever $A,B\in\CC$ and $A\neq B$. Assume that $A,B\in\CC$ are such that there exists a homeomorphism
$h:X\longrightarrow X$ such that $h[D_A]=D_B$. We will show that $A=B$. First we will show that $h[W]\subseteq\cl(W)$.

Let $O=h[W]\setminus\cl(W)$. Since $O\subseteq h[W]$, we have
$h^{-1}[O]\subseteq h^{-1}[h[W]]=W$. Therefore
$$
h^{-1}[O\cap E]=h^{-1}[O]\cap h^{-1}[E]\subseteq W\cap h^{-1}[E]\subseteq W\cap
D_A=D.
$$
But $D$ is a countable $\Gd$ subset of $W$ by construction, so
every subset of $D$ is also $\Gd$ in $W$. In particular $h^{-1}[O\cap E]$ is $\Gd$ in $W$, hence in $h^{-1}[O]$. It follows that $O\cap E$ is a $\Gd$ subset of $O$, which implies $O=\varnothing$.

Notice that $h[W]\subseteq\cl(W)$ implies $h[W]\subseteq W$, because $W$ is regular open and $h$ is a homeomorphism. A similar argument shows that $h^{-1}[W]\subseteq W$. Therefore $h[W]=W$, which implies $h[\cl(W)\setminus W]=\cl(W)\setminus W$. Hence $h[A]=B$, which concludes the proof.
\end{proof}

\begin{corollary}\label{meagerfiltersmanytypes}
Let $\FF$ be a meager filter. Then $\FF$ has $\cccc$ types of countable dense subsets.
\end{corollary}
\begin{proof}
It will be enough to show that every filter has the $\Po$. This is trivial if
$\FF=\COF$, so assume that $\FF\supsetneq\COF$. Let $U$ be an uncountable open subset of $\FF$. In particular $U\neq\varnothing$, so $[s]\cap\FF\subseteq U$ for some $s\in 2^{<\omega}$. Now pick any coinfinite $z\in\FF$ such that $z\upharpoonright\dom(s)=s$. It is easy to see that $[s]\cap (z\!\uparrow)$ is a copy of $2^\omega$ contained in $U$.
\end{proof}

\section{Non-meager P-subfilters}

Given a function $f:\omega\longrightarrow\omega$ and $\XX\subseteq\PP(\omega)$, define
$$
f(\XX)=\{x\subseteq\omega:f^{-1}[x]\in\XX\}.
$$
Recall that a function $f:\omega\longrightarrow\omega$ is \emph{finite-to-one} if it is surjective and $f^{-1}(n)$ is finite for every $n\in\omega$. It is easy to check that if $f$ is finite-to-one, then $f(\FF)$ is a filter
(respectively ultrafilter) whenever $\FF$ is a filter (respectively
ultrafilter).

We will make use of the following well-known theorem. Recall that $\uuuu<\gggg$ holds, for example, in Miller's model (see \cite[Section 11.9]{blass}).
\begin{theorem}\label{coherence}
Assume $\uuuu<\gggg$. Then there exists a $\mathsf{P}$-point $\UU$ such that for every non-meager filter $\FF$ there exists a finite-to-one $f:\omega\to\omega$ such that $f(\FF)=f(\UU)$.
\end{theorem}
\begin{proof}
Let $\UU$ be any ultrafilter generated by a set $\XX\subseteq\PP(\omega)$ such that $|\XX|<\gggg$. Notice that $\UU$ is a $\mathsf{P}$-point because $\gggg\leq\dddd$ by \cite[Proposition 6.27]{blass}, and every ultrafilter generated by less than $\dddd$ sets is a $\mathsf{P}$-point by \cite[Theorem 9.25]{blass}. The desired conclusion follows from the proof of \cite[Theorem 9.16]{blass}.
\end{proof}

\begin{proposition}\label{preserve}
Let $\FF$ be a filter, and let $f:\omega\longrightarrow\omega$ be finite-to-one. Notice that $\XX=\{f^{-1}[x]:x\in\FF\}$ has the finite intersection property, and let $\GG$ be the filter generated by $\XX$.
\begin{enumerate}
\item\label{Ppreimage} If $\FF$ is a $\mathsf{P}$-filter then $\GG$ is a $\mathsf{P}$-filter.
\item\label{nonmeagerpreimage} If $\FF$ is non-meager then $\GG$ is non-meager.
\end{enumerate}
\end{proposition}
\begin{proof}
The straightforward proof of (\ref{Ppreimage}) is left to the reader. To show that (\ref{nonmeagerpreimage}) holds, assume that
$\FF$ is non-meager. Fix a partition $\Pi$ of $\omega$ into finite sets and let $\Pi=\{I_k:k\in\omega\}$ be an injective enumeration. In order to show that $\GG$ is non-meager, by Theorem \ref{talagrandcriterion}, it will be enough to show that there exists $z\in\GG$ such that $z\cap I_k=\varnothing$ for infinitely many $k$.

Since $f$ is finite-to-one, there exists a sequence $k_0<k_1<\cdots$ of natural numbers such that $f[I_{k_m}]\cap f[I_{k_n}]=\varnothing$ whenever $m\neq n$. Let $\Pi'=\{J_n:n\in\omega\}$ be a partition of $\omega$ into finite sets such that $f[I_{k_n}]\subseteq J_n$ for each $n$. By Theorem \ref{talagrandcriterion}, there exists $x\in\FF$ such that $x\cap J_n=\varnothing$ for infinitely many $n$. It is easy to check that $z=f^{-1}[x]$ is the desired element of $\GG$.
\end{proof}

\newpage

\begin{theorem}\label{Psubfilter}
Assume $\uuuu<\gggg$. Then every non-meager filter has a non-meager $\mathsf{P}$-subfilter.
\end{theorem}
\begin{proof}
Let $\UU$ be the $\mathsf{P}$-point given by Theorem \ref{coherence}. Fix a non-meager filter $\FF$. Then there exists a finite-to-one $f:\omega\to\omega$ such that $f(\FF)=f(\UU)$. Let 
$$
\XX=\{f^{-1}[x]:x\in f(\FF)\}.
$$
Notice that $\XX\subseteq\FF$ by the definition of $f(\FF)$. Let $\GG$ be the subfilter of $\FF$ generated by $\XX$. It is easy to check that $f(\UU)$ is a $\mathsf{P}$-point, hence a non-meager $\mathsf{P}$-filter. Since $\XX=\{f^{-1}[x]:x\in f(\UU)\}$, it follows from Proposition \ref{preserve} that $\GG$ is a non-meager $\mathsf{P}$-filter.
\end{proof}

\begin{corollary}\label{cbdense}
Assume $\uuuu<\gggg$. Then every non-meager filter has a dense $\CB$ subspace.
\end{corollary}
\begin{proof}
Simply apply Theorem \ref{main}.
\end{proof}

It is natural to ask whether the above theorem and corollary actually hold in $\ZFC$. In Section 8, we will show that this is not the case for Theorem \ref{Psubfilter} (see Corollary \ref{kunenstrange}). However, we do not know the answers to the following questions.

\begin{question}\label{cbzfc}
Is it possible to prove in $\ZFC$ that every non-meager filter has a dense $\CB$ subspace?
\end{question}

\begin{question}\label{cbdensePsubfilter}
For a filter $\FF$, is having a non-meager $\mathsf{P}$-subfilter equivalent to having a dense $\CB$ subspace?
\end{question}

Assume that $D$ is a dense $\CB$ subspace of a filter $\FF$, and let $\GG$ denote then the subfilter of $\FF$ generated by $D$. Notice that $\GG$ is non-meager (because it has a dense Baire subspace). Hence, in order to answer Question \ref{cbdensePsubfilter}, one might try to show that $\GG$ is necessarily a $\mathsf{P}$-filter. The following proposition shows that this approach is not going to work.

\begin{proposition}
Assume $\MAc$. Then there exists a dense $\CB$ subspace $\Aa$ of $2^\omega$ with the finite intersection property such that the filter generated by $\Aa$ is not a $\mathsf{P}$-filter.
\end{proposition}
\begin{proof}
Let $\Aa$ be the independent family given by Theorem \ref{CBdenseindependent}, and let $\FF$ the filter generated by $\Aa$. Fix $\BB\in [\Aa]^\omega$. We claim that $\BB$ has no pseudointersection in $\FF$. Assume, in order to get a contradiction, that $z\in\FF$ is a pseudointersection of $\BB$. Since $\FF$ is generated by $\Aa$, there exists a non-empty $F\in [\Aa]^{<\omega}$ such that $\bigcap F\subseteq^\ast z$. Fix $x\in\BB\setminus F$. Notice that $z\subseteq^\ast x$ because $z$ is a pseudointersection of $\BB$. Therefore $\bigcap F\subseteq^\ast x$. It follows that $\bigcap F\cap(\omega\setminus x)$ is finite, which contradicts the fact that $\Aa$ is an independent family.
\end{proof}

\section{Two consistent characterizations}

In this section, we will combine several of the results discussed so far to show that, consistently, Theorem \ref{main} can be improved. Given a subspace $X$ of $2^\omega$ and a countable dense subset $D$ of $X$, define the \emph{relative type} of $D$ to be 
$$
\{h[D]:h\text{ is a homeomorphism of }2^\omega\text{ such that }h[X]=X\}.
$$
Notice that a space is $\RCDH$ if and only if it has exactly one relative type of countable dense subsets.

\begin{theorem}\label{anotherequivalent}
Assume $\uuuu<\gggg$. Then the following can be added to the list of equivalent conditions in Theorem \ref{main}.
\begin{enumerate}
\item[(9)] $\FF$ has less than $\cccc$ relative types of countable dense subsets.
\item[(10)] $\FF$ has less than $\cccc$ types of countable dense subsets.
\end{enumerate}
\end{theorem}
\begin{proof}
It is clear that $(\ref{cdhp})\rightarrow (9)\rightarrow (10)$. We will finish the proof by showing that $(10)\rightarrow (\ref{completelybaire})$. Suppose that $\FF$ is not $\CB$. If $\FF$ is meager then $\FF$ has $\cccc$ types of countable dense subsets by Corollary \ref{meagerfiltersmanytypes}, so assume that $\FF$ is non-meager. Then $\FF$ has a dense $\CB$ subspace by Corollary \ref{cbdense}.
Therefore $\FF$ has $\cccc$ types of countable dense subsets by Theorem \ref{hrusakvanmillctypes}. 
\end{proof}
\begin{corollary}\label{ctypesunderulessg}
Assume $\uuuu<\gggg$. Then there exists an ultrafilter with $\cccc$ types of countable dense subsets.
\end{corollary}

It is natural to ask whether the above theorem and corollary hold in $\ZFC$. Observe that, by Theorem \ref{main}, the answer to the following question is ``no'' if and only if Theorem \ref{anotherequivalent} holds in $\ZFC$. See also Corollary \ref{ctypesundermac}.

\begin{question}
Is it consistent that there exists a filter with $\kappa$ types of countable dense subsets, where $1<\kappa<\cccc$?
\end{question}

\begin{question}\label{ctypesquestion}
Is it possible to construct in $\ZFC$ an ultrafilter (or a non-meager filter) with $\cccc$ types of countable dense subsets?
\end{question}

\section{Another ultrafilter with $\cccc$ types of countable dense subsets}

In this section, we construct an ultrafilter with $\cccc$ types of countable dense subsets under $\MAc$ (see Corollary \ref{ctypesundermac}). Notice that this result does not overlap with Corollary \ref{ctypesunderulessg} because the assumptions $\uuuu<\gggg$ and $\MAc$ are incompatible (since $\mathsf{cov}(\BB)\leq\rrrr\leq\uuuu$ by \cite[Propositions 5.19 and 9.7]{blass}, and $\MAc$ is equivalent to $\mathsf{cov}(\BB)=\cccc$ by \cite[Theorem
7.13]{blass}). We will need the following preliminary lemma, inspired by \cite{kunenweak}.

\begin{lemma}\label{nonP}
Let $\Aa$ be an independent family. Then there exists an ultrafilter $\UU$ extending $\Aa$ that is not a $\mathsf{P}$-point.
\end{lemma}
\begin{proof}
Without loss of generality, assume that $\Aa$ is infinite. Fix $\BB\in [\Aa]^\omega$. It is easy to check that 
$$
\XX=\Aa\cup\{\omega\setminus x:x\subseteq^\ast y\text{ for every }y\in\BB\}
$$
has the finite intersection property. Let $\UU$ be any ultrafilter extending $\XX$. It is clear that $\BB$ has no pseudointersection in $\UU$.
\end{proof}

\begin{theorem}\label{CBdenseindependent}
Assume $\MAc$. Then there exists an independent family $\Aa$ that is dense in $2^\omega$ and $\CB$.
\end{theorem}

\begin{proof}
Enumerate as $\{Q_\eta:\eta\in\cccc\}$ all copies of $\QQQ$ in $2^\omega$, making sure to list each one cofinally often. We will construct $\Aa_\xi$ for every $\xi\in\cccc$ by transfinite recursion. In the end, set
$\Aa=\bigcup_{\xi\in\cccc}\Aa_\xi$. By induction, we will make sure that the following requirements are satisfied.
\begin{enumerate}
\item $\Aa_\eta\subseteq\Aa_\xi$ whenever $\eta\leq\xi<\cccc$.
\item $\Aa_\xi$ is an independent family for every $\xi\in\cccc$.
\item $|\Aa_\xi|<\cccc$ for every $\xi\in\cccc$.
\item If $Q_\eta\subseteq\Aa_\eta$ then there exists $z\in\Aa_\xi$ such that $z\in\cl(Q_\eta)\setminus Q_\eta$.
\end{enumerate}

Start by letting $\Aa_0$ be a countable independent family that is dense in $2^\omega$. Take unions at limit stages. At a successor stage $\xi=\eta+1$, assume that $\Aa_\eta$ is given. First assume that $Q_\eta\nsubseteq \Aa_\eta$. In this case, simply set $\Aa_\xi=\Aa_\eta$. Now assume that $Q_\eta\subseteq\Aa_\eta$. Apply Lemma $\ref{keyncbdensecb}$ with $\Aa=\Aa_\eta$ and $Q=Q_\eta$ to get $z\in\cl(Q_\eta)\setminus Q_\eta$ such that $\Aa_\eta\cup\{z\}$ is an independent family. Finally, set $\Aa_\xi=\Aa_\eta\cup\{z\}$.
\end{proof}

\begin{lemma}\label{keyncbdensecb} Assume $\MAc$. Let $\Aa$ be an independent family such that $|\Aa|<\cccc$. Let $Q\subseteq\Aa$ be crowded. Then there exists $z\in\cl(Q)\setminus Q$ such that
$\Aa\cup\{z\}$ is an independent family.
\end{lemma}
\begin{proof}
Consider the countable poset 
$$
\PPP=\{s\in 2^{<\omega}:\text{there exist }q\in Q\text{ and }n\in\omega\text{ such that }s=q\upharpoonright n\},
$$
with the natural order given by reverse inclusion.

For every $\sigma\in [\Aa]^{<\omega}$, $\nu:\sigma\longrightarrow 2$,
$\varepsilon\in 2$ and $\ell\in\omega$, define
$$
D_{\sigma,\nu,\varepsilon,\ell}=\{s\in\PPP:\text{there exists
}i\in\dom(s)\setminus\ell 
$$
$$
\text{ such that }s(i)=\varepsilon\text{ and }x(i)=\nu(x)\text{ for every
}x\in\sigma\}.
$$
Using the fact that $Q$ is crowded and $Q\subseteq\Aa$, one can show that each $D_{\sigma,\nu,\varepsilon,\ell}$ is dense in $\PPP$. For every $q\in Q$, define
$$
D_q=\{s\in\PPP:\text{there exists }i\in\dom(s)\text{ such that }s(i)\neq
q(i)\}.
$$
It is easy to see that each $D_q$ is dense in $\PPP$.

Since $|\Aa|<\cccc$ and $Q\subseteq\Aa$, the collection of dense sets
$$
\DD=\{D_{\sigma,\nu,\varepsilon,\ell}:\sigma\in
[\Aa]^{<\omega},\nu:\sigma\longrightarrow 2,\varepsilon\in
2,\ell\in\omega\}\cup\{D_q:q\in Q\}
$$
has also size less than $\cccc$. Therefore, by $\MAc$, there exists a
$\DD$-generic filter $G\subseteq\PPP$. Let $z=\bigcup G\in 2^\omega$. The dense sets of the form $D_{\sigma,\nu,\varepsilon,\ell}$ ensure that $\Aa\cup\{z\}$ is an independent family. The definition of $\PPP$ guarantees that $z\in\cl(Q)$. Finally, the dense sets of the form $D_q$ guarantee that $z\notin Q$.
\end{proof}

\begin{corollary}\label{CBdensenotCB}
Assume $\MAc$. Then there exists an ultrafilter that is not $\CB$ but has a dense $\CB$ subspace.
\end{corollary}
\begin{proof}
Let $\Aa$ be the independent family given by Theorem \ref{CBdenseindependent}. By Lemma \ref{nonP}, there exists an ultrafilter $\UU\supseteq\Aa$ that is not a $\mathsf{P}$-point. It is clear that $\Aa$ is a dense $\CB$ subspace of $\UU$. To see that $\UU$ is not $\CB$, use Theorem \ref{main}.
\end{proof}

\begin{corollary}\label{ctypesundermac}
Assume $\MAc$. Then there exists an ultrafilter with $\cccc$ types of countable dense subsets.
\end{corollary}
\begin{proof}
Let $\UU$ be the ultrafilter given by Corollary \ref{CBdensenotCB}. To see that $\UU$ has $\cccc$ types of countable dense subsets, apply Theorem \ref{hrusakvanmillctypes}.
\end{proof}

Given Theorem \ref{CBdenseindependent}, the following question seems natural. Notice that if the answer to Question \ref{independentquestion} is ``yes'' then the answer to Question \ref{ctypesquestion} is also ``yes'' (see the proof of Corollary \ref{ctypesundermac}).
\begin{question}\label{independentquestion}
Is it possible to construct in $\ZFC$ an independent family that is dense in $2^\omega$ and $\CB$?
\end{question}

\section{An ultrafilter with no non-meager P-subfilters}

By Theorem \ref{Psubfilter}, the statement ``Every non-meager filter contains a non-meager $\mathsf{P}$-subfilter'' is consistent. In this section, we show that the negation of this statement is also consistent (see Corollary \ref{kunenstrange}).

\begin{theorem}
Assume $\Diamond$. Then there exists an ultrafilter $\UU$ such
that whenever $\XX\subseteq\UU$ either there exists a countable subset of $\XX$ with no pseudointersection in $\UU$ or $\XX$ has a pseudointersection.
\end{theorem}

\begin{proof}
Let $\PP(\omega)=\{z_\xi:\xi\in\omega_1\}$ be an enumeration such that the
following conditions hold, where $\Lim=\{\xi\in\omega_1:\xi\text{ is a limit
ordinal}\}$.
\begin{itemize}
\item $\BB_\xi=\{z_\eta:\eta\in\xi\}$ is a Boolean subalgebra of $\PP(\omega)$ for every $\xi\in\Lim$.
\item $\BB_\omega=\FIN\cup\COF$.
\end{itemize}
By $\Diamond$, there exists a sequence $\langle\XX_\xi:\xi\in\Lim\rangle$ such
that $\XX_\xi\subseteq\BB_\xi$ for every $\xi\in\Lim$ and whenever
$\XX\subseteq\PP(\omega)$ the set $\{\xi\in\Lim:\XX_\xi=\XX\cap\BB_\xi\}$ is
stationary in $\omega_1$.

We will construct $\UU_\xi\subseteq\BB_\xi$, and $P_\xi\subseteq\Lim\cap\xi$ for
every $\xi\in\Lim$ by transfinite recursion. In the end, set
$\UU=\bigcup_{\xi\in\Lim}\UU_\xi$ and $P=\bigcup_{\xi\in\Lim}P_\xi$. Also define
the ideal
$$
\II_\xi=\{z\subseteq\omega:\text{there exists }k\in\omega\text{ and
}\{\eta_0,\ldots,\eta_{k-1}\}\in [P_\xi]^k\text{ such that }
$$
$$
z\subseteq^\ast w_0\cup\cdots\cup w_{k-1}\text{ whenever
}(w_0,\ldots,w_{k-1})\in\XX_{\eta_0}\times\cdots\times\XX_{\eta_{k-1}}\}
$$
for every $\xi\in\Lim$. By induction, we will make sure that the following
requirements are satisfied.
\begin{enumerate}
\item\label{ultraonA} $\UU_\xi$ is an ultrafilter on $\BB_\xi$ for every
$\xi\in\Lim$.
\item $\UU_\eta\subseteq\UU_\xi$ whenever $\eta,\xi\in\Lim$ and $\eta\leq\xi$.
\item $P_\eta\subseteq P_\xi$ whenever $\eta,\xi\in\Lim$ and $\eta\leq\xi$.
\item\label{promise} $\UU_\xi\cap \II_\xi=\varnothing$ for every $\xi\in\Lim$.
\end{enumerate}

Start by setting $\UU_\omega=\COF$ and $P_\omega=\varnothing$. Observe that $\II_\omega=\FIN$. Let
$\xi\in\Lim$ be such that $\xi>\omega$, and assume that $\UU_\eta$ and $P_\eta$
are given for every limit $\eta<\xi$. If $\xi$ is a limit of limit ordinals, let $\UU_\xi=\bigcup_{\eta\in\Lim\cap\xi}\UU_\eta$ and
$P_\xi=\bigcup_{\eta\in\Lim\cap\xi}P_\eta$. Now assume that $\xi=\eta+\omega$
for some $\eta\in\Lim$. First define $P_\xi=P_\eta\cup\{\eta\}$, if doing so
does not violate condition $(\ref{promise})$. Otherwise, let $P_\xi=P_\eta$.
Then extend $\UU_\eta$ to an ultrafilter $\UU_\xi$ on $\BB_\xi$, by deciding
whether $z_{\eta+n}$ belongs to $\UU_\xi$ for every $n\in\omega$ so that
condition $(\ref{promise})$ is preserved. This concludes the construction.

To verify that $\UU$ has the desired properties, let $\XX\subseteq\UU$. Then the
set
$$
S=\{\eta\in\Lim:\XX\cap\BB_\eta=\XX_\eta\}
$$
is stationary in $\omega_1$. If $S\cap P\neq\varnothing$, say $\eta\in S\cap P$,
then condition $(\ref{promise})$ guarantees that $\XX_\eta\subseteq\XX$ has no
pseudointersection in $\UU$.

So assume that $S\cap P=\varnothing$. Then, for each $\eta\in S$, there must be
$z\in\UU_\eta$, $k\in\omega$ and $\sigma=\{\eta_0,\ldots,\eta_{k-1}\}\in
[P_\eta]^k$ such that
$$
z\subseteq^\ast w_0\cup\cdots\cup w_{k-1}\cup w
$$ 
whenever
$(w_0,\ldots,w_{k-1},w)\in\XX_{\eta_0}\times\cdots\times\XX_{\eta_{k-1}}
\times\XX_\eta$. By the Pressing-Down Lemma and the Pigeonhole Principle, there exists an uncountable $S'\subseteq S$ such that the same $z$, $k$ and $\sigma$ work for every $\eta\in S'$. Fix such $S'$, $z$, $k$ and $\sigma$. By condition (\ref{promise}), there exists $(w_0,\ldots,w_{k-1})\in\XX_{\eta_0}\times\cdots\times\XX_{\eta_{k-1}}$ such that $x=z\setminus (w_0\cup\cdots\cup w_{k-1})$ is infinite. Using the fact that $S'\subseteq S$ is uncountable, it is easy to check that $x\subseteq^\ast w$ for every $w\in\XX$.
\end{proof}

\begin{corollary}\label{kunenstrange}
Assume $\Diamond$. Then there exists an ultrafilter $\UU$ such
that whenever $\FF$ is a subfilter of $\UU$, either $\FF$ is meager or $\FF$ is not a $\mathsf{P}$-filter.
\end{corollary}

Whenever one proves that a certain statement is a consequence of $\Diamond$, it is natural to wonder whether the same statement follows simply from $\CH$. The following question is a particular instance of this general principle.
\begin{question}
Can the assumption of $\Diamond$ be weakened to $\CH$ in Corollary \ref{kunenstrange}?
\end{question}

\section{Acknowledgements}

The authors would like to acknowledge an early contribution due to David Milovich. In fact, long before we obtained Theorem \ref{main}, he was able to show that an ultrafilter that is $\RCDH$ must be a $\mathsf{P}$-point (private communication to the second-listed author). They also thank the anonymous referee for several valuable suggestions.

\end{document}